\documentclass[a4paper,11pt]{amsart}
\usepackage[english]{babel}
\usepackage[latin1]{inputenc}
\usepackage{amssymb}
\usepackage{amsmath}
\usepackage{amscd}
\usepackage[colorlinks=true,linkcolor=blue,urlcolor=black]{hyperref}
\usepackage{float}
\usepackage{stmaryrd}

\usepackage{amsfonts}
\usepackage{latexsym}
\usepackage[all]{xy}
\usepackage{graphicx}
\usepackage{pb-diagram}
\usepackage{verbatim}
\usepackage{anysize}
\usepackage{cancel}
\usepackage{color,xcolor}
\usepackage[normalem]{ulem} 

\definecolor{mypurple}{RGB}{128,37,92}
\colorlet{ColorRelleno}{mypurple!50!white}

\usepackage[cal=boondox,calscaled=.96]{mathalfa}

\usepackage{fancyhdr}

\marginsize{2cm}{2cm}{2.5cm}{2.5cm}

\def\QQ{\mathbb{Q}}
\def\NN{\mathbb{N}}

\def\lbr{\llbracket}
\def\rbr{\rrbracket}

\newcommand{\N}{\mathbb N}

\newcommand{\Mp}{\mathfrak p}
\newcommand{\mm}{\mathfrak m}
\newcommand{\Ma}{\mathfrak a}

\newcommand{\M}{\mathcal M}
\newcommand{\bfx}{{\bf x}}
\newcommand{\bfb}{{\bf b}}

\newtheorem{teo}{Theorem}[section]
\newtheorem{prop}[teo]{Proposition}
\newtheorem{cor}[teo]{Corollary}
\newtheorem{lem}[teo]{Lemma}
{\theoremstyle{definition}
	\newtheorem{Def}[teo]{Definition}
	\newtheorem{nota}[teo]{Remark}}

\newtheorem{Def-Lem}[teo]{Definition-Lemma}
\newtheorem{parr}[teo]{ }

\newenvironment{proof1}[2]%
{\par\noindent{\it Proof of #1 }{#2}.  
	\nopagebreak\normalsize}%
{\hfill\linebreak[2]\hspace*{\fill}$\square$}

\DeclareMathOperator{\Spec}{\rm Spec}

\DeclareMathOperator{\Specmax}{\rm Specmax}

\DeclareMathOperator{\Leap}{\rm Leaps}

\DeclareMathOperator{\rank}{\rm rank}

\DeclareMathOperator{\End}{\rm End}

\DeclareMathOperator{\IDer}{\rm IDer}

\DeclareMathOperator{\Der}{\rm Der}

\DeclareMathOperator{\HS}{\rm HS}

\DeclareMathOperator{\Id}{\rm Id}

\DeclareMathOperator{\Sing}{\rm Sing}

\DeclareMathOperator{\Min}{\rm Min}

\DeclareMathOperator{\het}{\rm ht}

	\title{Finiteness of Leaps in the sense of Hasse-Schmidt of reduced rings}
\author{A. Bravo, Mar\'ia de la Paz Tirado Hern\'andez}
\thanks{The first author was partially supported by the Grant PID2022-138916NB-I00 funded by MCIN/AEI/10.13039/501100011033 and by ERDF A way of making Europe, and by 
		  the Spanish Ministry of Science and Innovation, through the ``Severo Ochoa'' Programme for Centres of Excellence in R\&D (CEX2019-000904-S).  The second author was partially supported by the Grant PID2020-114613GB-I00 funded by the Spanish Ministry of Science and Innovation and FJC2020-045783-I/AEI/10.13039/501100011033. Both authors were partially supported by 
		    the Madrid Government (Comunidad de Madrid - Spain) under the multiannual Agreement with UAM in the line for the Excellence of the University Research Staff in the context of the V PRICIT (Regional Programme of Research and Technological Innovation) 2022-2024.}

\keywords{Hasse-Schmidt derivation, Integrable derivation, Leap}
\subjclass[2020]{13N15}

\AtEndDocument{\bigskip{\footnotesize
		\noindent\textsc{Depto. Matem\'aticas, Facultad de Ciencias, Universidad Aut\'onoma de Madrid and Instituto de Ciencias Matem\'aticas, ICMAT,  CSIC-UAM-UC3M-UCM, Cantoblanco 28049 Madrid, Spain} \newline
		\textit{E-mail address}, A. Bravo: \texttt{ana.bravo@uam.es}; https://orcid.org/0000-0002-2933-1132
		\medskip
		
		\noindent\textsc{Depto. Matem\'aticas, Facultad de Ciencias, Universidad Aut\'onoma de Madrid, Cantoblanco 28049 Madrid, Spain. Instituto de Matem\'aticas (IMUS), Facultad de Matem\'aticas, Universidad de Sevilla, calle Tarfia s/n, 41012 Sevilla, Spain} \newline
		\textit{E-mail address}, M.P. Tirado Hern\'andez: \texttt{maria.tirado@uam.es}; https://orcid.org/0000-0002-8006-843X	}}

\begin{document}
	
		\begin{abstract}
	 We give sufficent conditions for a derivation of a  $k$-algebra $A$ of finite type  to be $\infty$-integrable in the sense of Hasse-Schmidt, when $A$ is a complete intersection,  or when $A$ is reduced and $k$ is a regular ring. As a consequence, we prove that, if  in addition  $A$ contains a field, then  the set of leaps of $A$ is finite along the minimal primes of   certain Fitting ideal  of $\Omega_{A/k}$. 
	\end{abstract}
	
	\maketitle
	

	\section{Introduction}

\color{black} Let $k$ be a ring and let $A$ be a commutative $k$-algebra   with unit. The $A$-module of derivations $
\Der_k(A)$ relies on    the  infinitesimal structure of $A$, and henceforth encodes information about its singularities if $A$ is Noetherian.  When ${\mathbb   Q}\subset A$,  examples in this line can be found in   the work of  Seidenberg about the extension of derivations to the integral closure of a Noetherian domain,  or the action of $
\Der_k(A)$ on the associate primes  of   $A$ or on the minimal primes of the singular locus of $A$   (when $A$ is of finite type over a field)  (see \cite{Se}, \cite{Se2}); also   in  Zariski's results  on the   completion  of a local ring having a derivation that maps the maximal ideal to the unit ideal (see \cite{Zariski}). Under quite general assumptions, if $k$ is a field, we have that   $\rank 
\Der_k(A)\leq \dim A$  (see \cite{Matsumura}, \cite{Molinelli}).

\medskip

Annoyingly enough,  these  results do not hold, in general, when the characteristic of the ring is positive, the obstruction being that, in this case, the $A$-module $
\Der_k(A)$ is simply too large.  A good alternative  is to consider the $A$-submodule of  
$\Der_k(A)$ that consists of the $\infty$-integrable derivations of $A$, $\IDer_k(A)$:  these are the derivations that can be extended to a  Hasse-Schmidt derivation  of $A$ over $k$ (of lenght $\infty$)\color{black}. The   results    above   have been   shown to hold in     prime   characteristic   when  $\Der_k(A)$ is replaced  by $\IDer_k(A)$  (see \cite{BK}, \cite{mat-intder-I}, \cite{Mo}, \cite{Se}).  We also refer to \cite{NaFor} where the good properties of   $\IDer_k(A)$ in positive characteristic, in contrast to those of $\Der_k(A)$,   are further explored.  

\medskip

 The notion of $\infty$-integrable derivation can be considered as a particular case of the following more general one. 
  For $m\in {\mathbb N}_{\geq 1} \cup \{\infty\}$,  a $k$-derivation $\delta:A\to A$ is said to be {\em $m$-integrable} if it extends up to a Hasse-Schmidt derivation $D=(\Id, D_1=\delta, \ldots, D_m)$ of $A$ over $k$ of length $m$  (see \cite{Na2}).  Analogously,  the set of $m$-integrable $k$-derivations,    $\IDer_k(A;m)$, is an $A$-submodule of $\Der_k(A)$.

\medskip 

Hence, there is a (non-increasing) chain of $A$-modules: 
\begin{equation}
	\label{sucesion_integrab} 
	\Der_k(A)=\IDer_k(A;1) \supseteq \IDer_k(A;2)\supseteq \ldots \supseteq \IDer_k(A;\infty),
\end{equation}
and we say that $A$ has a {\em leap} at $s\geq 2$ if the $(s-1)$-th inclusion of this chain is proper, i.e. if $\IDer_k(A;s-1)/\IDer_k(A;s)\neq 0$.    We will use $\Leap_k(A)$ to denote the set of   leaps of $A$.  If $\QQ\subset k$, or if $k$   is arbitrary but $A$ is $0$-smooth over $k$, then  $\Leap_k(A)=\emptyset$ (see \cite{mat-intder-I}), and hence $\Der_k(A)=\IDer_k(A)$. Thus $\Leap_k(A)$  interests us  when the characteristic is positive and, in particular, but non exclusively, when $A$ is   non-regular.  

\medskip
The prior discussion leads to two natural questions. On the one hand, it would be interesting to have some criterion to decide whether a given derivation $\delta\in \Der_k(A)$  is $\infty$-integrable or not. On the other hand, one may wonder    how the set $\Leap_k(A)$ is,   and whether it tells us something about the singularities of $A$. 

\medskip

In connection with the first question in \cite[Theorem 11]{mat-intder-I}   it was proven that if $A$ is a domain of finite type over a perfect field $k$, then any $k$-derivation which sends $A$ into its Jacobian ideal is $\infty$-integrable.   Furthermore,  in \cite[Proposition 2.2.1]{Na2C} this result was  generalized to the case in which  $k$ is any commutative ring and  $A=k[x_1,\ldots,x_n]/I$,  where $I$ is a principal ideal.   

 \medskip

Suppose we are given $D\in \HS_k(A;m)$. The problem of extending $D$ to some $D'\in \HS_k(A;m+1)$ can be translated into a problem of finding a solution to a suitable linear system of equations over the ring $A$. In \cite[Theorem 11]{mat-intder-I} the fact that $A$ is a domain is used to  treat such  system of equations on $K(A)$, the ring of fractions of $A$, and there, Cramer's rule can be  applied.  Here, we work  in a more general setting, dropping the assumption of $A$ being a domain, or even that of $k$ being a field.   Thus, we need to work    out  some criteria  that guarantees the existence of a solution to a  linear equations over arbitrary (commutative) rings. This leads us to generalize the previous results in  different  ways.  First, with no extra hypotheses on $k$:

\medskip 

\noindent{\bf Theorem   \ref{Teo-DerIntInterseccionCompleta}.}
{\em Let  $k$ be a commutative ring with unit, set      $R=k[x_1,\ldots, x_n]$, let  $I=\langle f_1,\ldots, f_r\rangle$  and set  $A=R/I$. Let  $\delta\in \Der_k(A)$ and let $J_r$ be the $(n-r)$--Fitting ideal  of $\Omega_{A/k}$. If  $\delta(A)\subset J_r$, then  $\delta\in \IDer_k(A)$. In  particular, $J_r\Der_k(A)\subset \IDer_k(A)$. }

\medskip

The next result requires $k$ to be regular, but in exchange, it gives us a sufficient condition for integrability, for instance, in the case of equidimensional reduced $k$-algebras of finite type:

\medskip

\noindent {\bf Theorem   \ref{Teo-CasoNoEquidimensional}.} 
{\em Let $k$ be a regular ring,  set $R=k[x_1,\ldots, x_n]$ and let  $I\subset R$ be a radical ideal with  $r=\max \{\het(P)\ | \ P\in \Min(I)\}$.  Let  $A=R/I$ and  let  $J_r$ be the $(n-r)$--Fitting ideal  of $\Omega_{A/k}$.  Then,   $J_r\Der_k(A)\subset \IDer_k(A)$.}

In fact, we will see that 	Theorems \ref{Teo-DerIntInterseccionCompleta} and  \ref{Teo-CasoNoEquidimensional} also hold after replacing $R=k[x_1,\ldots, x_n]$ by $\widetilde{R}=k\lbr x_1,\ldots, x_n\rbr$ (see Remark     \ref{series_integrabilidad}).

\medskip

  Regarding to the study of the set $\Leap_k(A)$, 
 the second author  has shown  that, for a commutative ring of prime characteristic $p$,   the leaps of $A$   can only occur at powers of $p$ (see \cite{Ti1}). Furthermore,  the set of leaps remains the same under some base changes, such as separable ring extensions over a field of positive characteristic (see \cite{Ti2}). However, algebroid curves with the same semigroup may have different sets of leaps (see \cite{TiPhD}).
 
 \medskip

The appearance of  leaps is a  phenomenon that can only occur on rings of positive characteristic. It 
 would be interesting to explore  further connections with other pathologies  of singularities of varieties in characteristic $p>0$. To start with, it is quite natural to wonder about the stability of sequence (\ref{sucesion_integrab}) at least in the case of algebraic varieties over perfect fields.

 \medskip

 More precisely, for a given set of powers of $p$, $\Lambda$, it is not hard to find examples of rings $A$ for which $\Leap_k(A)=\Lambda$:   if ${\mathbb F}_p\subset k$, and      $A=k[x_1,\ldots, x_n,\ldots]/\langle x_1^{p^{m_1}},\ldots, x_n^{p^{m_n}},\ldots\rangle$, then $\Leap_k(A)=\cup_i\{p^{m_i}\}$. 
However, one might  wonder    whether $\Leap_k(A)$ is finite  assuming some natural finiteness conditions on $A$.  

\medskip

This issue was addressed in   \cite{NRT},  where it was proven that if $A$ is the coordinate ring of an irreducible affine curve over a perfect field $k$ of prime characteristic with geometrically unibranch singularities, then   $\Leap_k(A)$ is finite. Here  we generalize this result. More precisely, we prove: 
 
\medskip

\noindent{\bf Theorem \ref{Teo-AnillosConSaltosFinitos}. }
{\em Let $k$ be a Noetherian ring containing a field, let $R=k[x_1,\ldots, x_n]$ and let  $I\subset R$ be an  ideal. Set  $A=R/I$ and let $\Mp\in \Spec(A)$ be  a minimal prime of  $J_r$, the   $(n-r)$--Fitting ideal of  $\Omega_{A/k}$.  Suppose that, at least,  one of the following conditions hold: 
	\begin{itemize}
		\item[1)] $I=\langle f_1,\ldots, f_r\rangle$. 
		\item[2)] $k$ is regular, $I$ is radical, and  $r=\max\{\het(P) \ | \ P\in \Min(I)\}$. 
	\end{itemize}
	Then the set $\Leap_k(A_{\Mp})$ is finite of cardinal  bounded by $\dim_K (\Der_k(A_\Mp)/{\mathfrak p}^M \Der_k(A_\Mp))$ where $K$ is the residue field of   $\Mp$,  and $M$ is the smallest positive integer so that  $\Mp^M\Der_k(A_\Mp)\subset \IDer_k(A_\Mp)$.  
}

\color{black}

\medskip

As a consequence: 

\medskip 

\noindent{\bf Corollary \ref{corolario_curvas}.} 
	{\em Let   $k$ be  a perfect field and let $A$ be a reduced $k$-algebra of finite type of  dimension 1. Then, $\Leap_k(A)$ is finite. }

\medskip 

 To conclude, regarding the case of formal power series ring we prove: 

\medskip

\noindent{\bf Theorem \ref{Teo-AnillosConSaltosFinitosFormal}.} {\em Let $k$ be a Noetherian ring containing a field, let $\widetilde{R}=k\lbr x_1,\ldots, x_n\rbr$ and let  $I\subset \widetilde{R}$ be an  ideal. Set  $\widetilde{A}=\widetilde{R}/I$  and let 
	$J_r:=J_r(\widetilde{A})$  be  the    $(n-r)$--Fitting ideal of  $\widetilde{\Omega}_{\widetilde{A}/k}$.  Suppose that the radical of $J_r$ is a maximal ideal ${\mathfrak m}\subset \widetilde{A}$ and that, in addition,  at least one of the following conditions hold: 
	\begin{itemize}
		\item[1)] $I=\langle f_1,\ldots, f_r\rangle$; 
		\item[2)] $k$ is regular, $I$ is radical and  $r=\max\{\het(P) \ | \ P\in \Min(I)\}$. 
	\end{itemize}
	Then the set $\Leap_k(\widetilde{A})$ is finite of cardinal bounded by $d:=\dim_K (\Der_k(A)/{\mathfrak m}^M \Der_k(\widetilde{A}))$ where $K$ is the residue field of   ${\mathfrak m}$,  and $M$ is the smallest positive integer so that  ${\mathfrak m}^M\Der_k(\widetilde{A})\subset \IDer_k(\widetilde{A})$.}

\color{black}

\medskip

The paper is organized as follows: in section 2 we recall the definitions of Hasse-Schmidt derivations and integrability,  and review some basic results. In section 3, we give a sufficient condition for  a local Noetherian $k$-algebra containing a field to have a finite number of leaps. In section 4 we study the existence of suitable sets of   {\em generic generators}  of a radical ideal, $I\subset R=k[x_1,\ldots, x_n]$, 
 in the presence of some  non-vanishing conditions of the  Fitting ideals of $\Omega_{A/k}$, for  $A=R/I$.  This result will be used  in   section 5, where we prove two of the main results of this paper, Theorems \ref{Teo-DerIntInterseccionCompleta} and \ref{Teo-CasoNoEquidimensional}. In    section 6, we address Theorem \ref{Teo-AnillosConSaltosFinitos}. Throughout the  paper, all rings (and algebras) are assumed to be
commutative   with unit.

\medskip
 
{\em Acknowledgments:} The authors profited from conversations with
  L. Narv\'aez Macarro. 
  
	\section{Hasse--Schmidt derivations}
	
	In this section, we recall the main definitions of the theory of
	Hasse-Schmidt derivations and review some basic results. From now on, $A$ will denote a commutative $k$-algebra with unit. We denote
	$\overline{\mathbb N}:=\mathbb N \cup \{\infty\}$ and, for each
	integer $m\geq 1$, we will write $A\lbr t\rbr_m:=A\lbr t\rbr/\langle
	t^{m+1}\rangle$ and $A\lbr t\rbr_\infty:=A\lbr t\rbr$. General
	references for the definitions and results in this section are
	\cite[\S 27]{Ma} and \cite{Na2}.

	\begin{Def}\label{DefHS}
 A {\em Hasse-Schmidt derivation or  a   HS-derivation\footnote{The HS-derivations are also called {\em higher derivations}, see \cite[\S 27]{Ma}.} of $A$ over
		$k$  of length $m\geq 1$} (resp. of length $\infty$) is a sequence \color{black}
		$D:=(D_0,D_1,\ldots, D_m)$ (resp. $D=(D_0,D_1,\ldots)$) of
		$k$-linear maps $D_\alpha:A\rightarrow A$, satisfying the conditions:
		$$
		\begin{array}{ccc}
			D_0=\Id_A,&\displaystyle D_\alpha(xy)=\sum_{i+j=\alpha} D_i(x)D_j(y),
		\end{array}
		$$
		for all $x,y\in A$ and for all $\alpha$.  We write $\HS_k(A;m)$ (resp.
		$\HS_k(A;\infty)=\HS_k(A)$) for the set of HS-derivations of $A$ (over
		$k$) of length $m$ (resp. $\infty$).

	\end{Def}
	
	The notion of HS-derivations (of length $\infty$) was introduced in \cite{H-S}. Any  HS-derivation $D$ of $A$ over $k$ can be interpreted as a power series $\sum_{\alpha=0}^m D_\alpha t^\alpha \in \End_k(A)\lbr t\rbr_m$. Actually, $\HS_k(A;m)$ is a (multiplicative) sub-group   of $\mathcal U(\End_k(A)\lbr t\rbr_m)$, the group of units of $\End_k(A)\lbr t\rbr_m$. The group operation in $\HS_k(A;m)$ is explicitly given by	$
	(D\circ D')_\alpha=\sum_{i+j=\alpha} D_i\circ D_j'$, and the identity element of $\HS_k(A;m)$ is $\mathbb{I}$ with $\mathbb I_0=\Id$ and $\mathbb I_\alpha=0$ for all $\alpha=1,\ldots, m$. In addition, any HS-derivation $D\in \HS_k(A;m)$ determines and is determined by
	the $k$-algebra homomorphism
	$$
	\varphi_D: a\in A \longmapsto a+\sum_{\alpha\geq 1}^m D_\alpha(a)t^\alpha \in
	A\lbr t \rbr_m.
	$$

	Given positive integers  $1\leq n \leq m$, there is a  group   homomorphism $\tau_{m,n}: \HS_k(A;m) \to \HS_k(A;n)$ corresponding to the obvious truncation  map. 
	We have the following identity of groups
	\begin{equation} \label{eq:limit-finite}
		\HS_k(A;\infty)
		= \lim_{\stackrel{\longleftarrow}{\substack{m}}} \HS(A;m).
	\end{equation}
	
	\medskip
 
	If $J$ is an ideal of $A$, a $k$-derivation $\delta:A\to A$ is called {\em $J$-logarithmic} if $\delta(J) \subset J$. The set of
	$J$-logarithmic $k$-derivations is an $A$-submodule of $\Der_k(A)$
	and  will be denoted by $\Der_k(\log J)$. Analogously, we define $J$-logarithmic HS-derivations of length $m$, which is a subgroup of $\HS_k(A;m)$. Namely, 
	\begin{Def}\label{Def-HSLogaritmica} Let $m\in \overline \NN$.
	We say that $D\in \HS_k(A;m)$ is {\em $J$-logarithmic} if $D_\alpha(J)\subseteq J$ for all $\alpha=0,\ldots, m$. The group of $J$-logarithmic HS-derivations of length $m$ is denoted by $\HS_k(\log J;m)$ and $\HS_k(\log J):=\HS_k(\log J;\infty)$.
	\end{Def}
	
	If $A$ is a finitely generated $k$-algebra, we may assume that $A$
	is the quotient of $R=k[x_1,\dots,x_n]$ by some ideal $J$. For every $m\in \overline \N$, there is a surjective group homomorphim:
	$$
	\HS_k(\log J;m) \longrightarrow \HS_k(A;m), 
	$$
	given by: 
	\begin{equation*}\label{delta barra}
		D\in \HS_k(\log J;m) \mapsto \overline{D} \in \HS_k(A;m)\mbox{ with }\overline{D_\alpha}(r+J) = D_\alpha(r) +J \mbox{ for all }r\in R,\ \alpha=0,\ldots, m.
	\end{equation*}

	\medskip

	\begin{Def}\label{Log-Int} Let $m\in \overline \N$ and $\delta:A\to A$ be a $k$-derivation. We say that $\delta$ is {\em $m$-integrable} (over $k$) if there is a  HS-derivation $D\in \HS_k(A; m)$ such that $D_1=\delta$.
	Any such $D$
	will be said to be an {\em $m$-integral} of $\delta$. The set of $m$-integrable $k$-derivations of $A$ is denoted by $\IDer_k(A; m)$ and $\IDer_k(A):=\IDer_k(A;\infty)$. 
	
	Let $J$ be an ideal of $A$. We say that $\delta$ is {\em $J$-logarithmically $m$-integrable} if there exists $D\in \HS_k(\log J;m)$ such that $D$ is an $m$-integral of
	$\delta$. We denote $\IDer_k(\log J;m)$ the set of $J$-logarithmically
	$m$-integrable derivations and $\IDer_k(\log
	J):=\IDer_k(\log J; \infty)$.
	\end{Def}
	
The sets $\IDer_k(A;m)$ and $\IDer_k(\log J;m)$ are $A$-submodules of $\Der_k(A)$. We have the following chains:
\begin{equation}
	\label{chain1}
\Der_k(A)=\IDer_k(A;1)\supseteq \IDer_k(A;2)\supseteq \ldots \supseteq \IDer_k(A),
\end{equation}
and 
\begin{equation}
	\label{chain2}
\Der_k(\log J)=\IDer_k(\log J;1)\supseteq \IDer_k(\log J;2)\supseteq \ldots \supseteq \IDer_k(\log J).  
\end{equation}

If $A$ is a finitely generated $k$-algebra, i.e., if $A=k[x_1,\ldots, x_n]/J$, there exist  surjective maps of $A$-modules: 
$$
\IDer_k(\log J;m)\ni\delta\longrightarrow \overline\delta\in \IDer_k(A;m) \ \  \mbox{where $\overline \delta(r+J)=\delta(r)+J$, \ } \forall m\in \overline\NN.
$$

\smallskip

If $\QQ\subset k$, then any derivation of $A$ is $\infty$-integrable,   and hence  $\Der_k(A)=\IDer_k(A)$ (see \cite[p.230]{mat-intder-I}). However, some of the  containments in sequence  (\ref{chain1})  could be strict   when the characteristic is positive (see \cite[Examples 1 to 3]{mat-intder-I}) and, in this case we say that $A$ has a {\em leap}. More precisely:

	\begin{Def}\label{leap} Let $s>1$ be an
		integer. We say that the $k$-algebra $A$ has a {\em leap} at $s$ if the
		inclusion $\IDer_k(A;s-1)\supsetneq \IDer_k(A;s)$ is proper. The set
		of leaps of $A$ over $k$ is denoted by $\Leap_k(A)$. 
		
		In addition, we say that a $k$-derivation {\em leaps} at $s\in \N_{\geq 2}$ if it is $(s-1)$-integrable but not $s$-integrable, i.e., if 
		$$\delta\in \IDer_k(A;s-1)\setminus \IDer_k(A;s).$$
		We say that a subset $\mathcal C\subset \Der_k(A)$ {\em produces a leap} at $s$ if there is some  $\delta\in \mathcal C$ leaping at $s$.
	\end{Def}

If $A$ is $0$-smooth over $k$, then any $k$-derivation is $\infty$-integrable (see \cite[Theorem 27.1]{Ma}), so $\IDer_k(A)=\Der_k(A)$. Thanks to this theorem we have the following result:
\begin{prop}\cite[Proposition 2.2.]{NRT}\label{Prop-SaltosRegular}
Let $k$ be a perfect field. Assume that $A$ is essentially of finite type over $k$. If $A$ is regular, then $\Der_k(A)=\IDer_k(A)$. In particular, $\Leap_k(A)=\emptyset$. 
\end{prop}

Hence, using \cite[Proposition 1.6]{NRT} it follows that:  
\begin{prop}\label{Prop-UnionSaltos}
Let  $k$ be a perfect field  and $A$ a finitely generated $k$-algebra. Then, 
$$
\Leap_k(A)=\bigcup_{P\in \Sing(A)} \Leap_k(A_{P}),
$$
where $\Sing(A)$ is the singular locus of $A$. 
\end{prop}

The next lemmas  will be also used throughout this paper.

\begin{lem}\label{Lem-HSInterseccionIdeales}
	Let $k$ be a commutative ring and $A$ a commutative  $k$-algebra. Let $J_1,\ldots, J_n$ be ideals of $A$ and $m\in \overline\NN$. Then, 
	$$
	\bigcap_{i=1}^n\HS_k(\log J_i;m)\subset \HS_k\left(\log \bigcap_{i=1}^n J_i; m\right).
	$$
\end{lem}

\begin{lem}\cite[Lemma 3.7]{NRT}\label{Lem-AsociadosDerivaciones} Let $k$ be a ring, $A$ a reduced Noetherian $k$-algebra and $P \in \Min(A)$. Then, for every $n \in \NN$, 
every $D \in  \HS_k(A; m)$ induces a HS-derivation in $\HS_k(A/P; m)$, i.e., $D\in \HS_k(\log P;m)$
	\end{lem} 
	
The next lemma is consequence of \cite[(1.2.10) and Corollary 2.3.5]{Na2}:

\begin{lem}\label{Cor-SJcDer infinito-integrable1}
Let $k$ be a ring and $A$ a finitely presented Noetherian $k$-algebra. Let $J$ an ideal of $A$ such that  $J\Der_k(A)\subset \IDer_k(A)$ and let  $P$ a minimal prime of $J$. Then
$P^N\Der_k(A_P)\subset \IDer_k(A_P)$ where  $N$ is the minimal integer such that  $P^N\subset JA_P$. 
\end{lem}

\section{A sufficient condition for $\Leap_k(A)$ to be finite }\label{SEC-FinitudSaltosLocales}

The purpose of this section is to give a proof of the following theorem:

\begin{teo}\label{Prop-Saltos finitos}  
	Let  $A$ be a   Noetherian $k$-algebra containing a field.  Suppose that  there is some maximal ideal ${\mathfrak m}\subset A$ and some integer $N\geq 1$ such that  
	\begin{equation}\label{condicionBIS}
		\mm^N\Der_k(A)\subset \IDer_k(A).
	\end{equation}  
 Let $K=A/\mm$. 	Then, 
	$$\#\Leap_k(A) \leq \dim_K(\Der_k(A)/\mm^N\Der_k(A)).$$
	In particular, if $\Der_k(A)$ is a finitely generated $A$-module satisfying condition (\ref{condicionBIS}), then $\Leap_k(A)$ is finite.
\end{teo}

\color{black}

To prove the theorem we will use Lemma \ref{Lem-Saltos finitos paso 1} stated below. First, a definition:

\begin{Def}
\label{Def-Conjunto span} 
For a given   finite collection of derivations,  $\delta_1,\ldots, \delta_r\in \Der_k(A)$, we will be considering the set: 
$$
\mathcal C(\delta_1,\ldots,\delta_r):=\left\{\sum_{i=1}^r \lambda_i\delta_i \ | \ \lambda_i\in \mathcal U(A) \cup \{0\}\right\}
$$
where  $\mathcal U(A)$ is the set of units in  $A$.  And for a given maximal ideal    ${\mathfrak m}\subset A$  we will be considering the set: 
$$
\mathcal D_{{\mathfrak m}}(\delta_1,\ldots,\delta_r):=\left\{\sum_{i=1}^r \lambda_i\delta_i \ | \ \lambda_i\in  (A\setminus {\mathfrak m})\cup \{0\}\right\}.
$$
\end{Def}

\color{black}
The next lemma generalizes results from \cite[\S 1.1]{NRT}. We include here the proof adapted to our (weaker) hypotheses to facilitate the reading of the paper.

\begin{lem}\label{Lem-Saltos finitos paso 1} Let $k$ be a ring and let $A$ be  a $k$-algebra. Suppose  that $\delta_1,\ldots, \delta_r\in \Der_k(A)$ are derivations such that for $i=1,\ldots, r$,   $\delta_i$ leaps at $s_i$,  with  $s_1<s_2<\ldots <s_r$. Then: 
	\begin{itemize}
		\item[(i)]  The only leaps produced by the set of derivations
		$$M=\mathcal C(\delta_1,\ldots, \delta_r)\subset \Der_k(A)$$ are $s_1,\ldots, s_r$.
		
		\item[(ii)] If there is some  maximal ideal ${\mathfrak m}\subset A$ and a positive integer $N$ such that ${\mathfrak m}^N  \Der_k(A) \subset \IDer_k(A)$, then, the only leaps produced by the set of derivations
		$$M_\mm=\mathcal D_{\mathfrak m}(\delta_1,\ldots, \delta_r)\subset \Der_k(A)$$ are $s_1,\ldots, s_r$.
	\end{itemize}
As a consequence, if $\Der_k(A)=M+L$ or if $\Der_k(A)=M_\mm+L$, with $L\subset \IDer_k(A)$,  then  $\Leap_k(A)=\{s_1,\ldots, s_r\}$.

\end{lem}

\begin{proof} (i)  From the hypothesis, $M$ leaps at $s_1,\ldots, s_r$. We will show that $M$ produces no more leaps.  Let  $\delta\in M\setminus \{0\}$. Then there are some   $\lambda_1,\ldots, \lambda_r\in \mathcal U(A) \cup \{0\}$ such that  $$\delta=\lambda_1\delta_1+\ldots+\lambda_r\delta_r.$$ 
	Let $m=\min \{e\in \{1,\ldots, r\} \ | \ \lambda_e\neq 0\}$. Then,  $\delta=\lambda_m\delta_m+\ldots+\lambda_r\delta_r$ is $(s_m-1)$-integrable. Now, if $\delta$ were $s_m$-integrable, then  $\lambda_m\delta_m$ would be  $s_m$-integrable but, since $\lambda_m\in \mathcal{U}(A)$, it would follow that $\delta_m$ is $s_m$-integrable, which contradicts the fact that $\delta_m$ leaps at $s_m$.
	
	\medskip
	
	(ii) From the hypothesis, $M_\mm$ leaps at $s_1,\ldots, s_r$. We will show that $M_\mm$ produces no more leaps.  Let  $\delta\in M_\mm\setminus \{0\}$. Then there are some   $\lambda_1,\ldots, \lambda_r\in( A\setminus {\mathfrak m} )\cup \{0\}$ such that  $$\delta=\lambda_1\delta_1+\ldots+\lambda_r\delta_r.$$ 
	Let $m=\min \{e\in \{1,\ldots, r\} \ | \ \lambda_e\neq 0\}$. Then,  $\delta=\lambda_m\delta_m+\ldots+\lambda_r\delta_r$ is $(s_m-1)$-integrable. Now, if $\delta$ were $s_m$-integrable, then  $\lambda_m\delta_m$ would be  $s_m$-integrable. Since $\lambda_m\in  A\setminus {\mathfrak m}$, there is some $a\in A\setminus {\mathfrak m}$ such that $a \lambda_m=1+ b$,  with $b\in {\mathfrak m}^N$. Hence, 
	$a\lambda_m\delta_m= \delta_m+ b\delta_m$
	from where it would follow that  $\delta_m$ is $s_m$-integrable, which contradicts the fact that $\delta_m$ leaps at $s_m$.
\end{proof}

\begin{proof1}{Theorem}{\ref{Prop-Saltos finitos}} 
	Since $A/\mm^N$ is local  complete and because $A$ contains a field, there is an isomorphic copy of $K=A/\mm$ in $A/\mm^N$  and $\Der_k(A)/\mm^N\Der_k(A)$ is a $K$-vector space. If the dimension of this $K$-vector space is not finite, then the result is trivial, so   assume that $\dim_K(\Der_k(A)/\mm^N\Der_k(A))=r<\infty$. Select  $\eta_1,\ldots, \eta_r \in \Der_k(A)$ so that  
	$\overline{\eta}_1,\ldots, \overline{\eta}_r\in \Der_k(A)/\mm^N\Der_k(A)$ form a basis of the later $K$-vector space. Then  
	$$
	\Der_k(A)={\mathcal D}_{\mathfrak m}(\eta_1,\ldots,\eta_r)+\mm^N \Der_k(A).
	$$
	If  $\eta_i$ does not leap for $i=1,\ldots, r$, then  $\Leap_k(A)=\emptyset$. We will show that otherwise  the cardinal of   $\Leap_k(A)$ is bounded above by $r$.

	Suppose that $r\leq \#  \Leap_k(A)$, and let  $s_1<\ldots< s_r$  be the first $r$ leaps.  Since the hypothesis is that $\mm^N\Der_k(A)\subseteq \IDer_k(A)$, there must be some  $\eta_i$ leaping at  $s_1$. Relabeling the $\eta_i$ if needed, we may assume that  $i=1$. Set $\delta_1:=\eta_1$. By  an inductive argument,  suppose     that we have found  $\delta_1,\ldots, \delta_j\in \Der_k(A)$ leaping at $s_1<s_2<\ldots<s_j$ so that  
	$$
	\Der_k(A)={\mathcal D}_{\mathfrak m}(\delta_1,\ldots, \delta_j, \eta_{j+1},\ldots,\eta_r)+\mm^N\Der_k(A),$$
	and,
	$$  \Der_k(A)/\mm^N\Der_k(A)=\langle \overline \delta_1,\ldots, \overline\delta_j, \overline \eta_{j+1},\ldots, \overline\eta_r\rangle_K,$$
	where $\overline\delta_i=\delta_i\mod \mm^N\Der_k(A)$.  Now let  $\delta_{j+1}\in \Der_k(A)$ leaping at  $s_{j+1}$, with $\overline\delta_{j+1}$ its image at  $\Der_k(A)/\mm^N\Der_k(A)$. Then, 
	$$
	\overline\delta_{j+1}=\overline\lambda_1\overline\delta_1+\ldots+\overline\lambda_j\overline\delta_j+\overline\lambda_{j+1}\overline\eta_{j+1}+\ldots+\overline\lambda_r\overline\eta_r
	$$ 
	with $\overline\lambda_i\in K$. For  $i=1,\ldots, r$, let $\lambda_i\in (A\setminus {\mathfrak m}) \cup \mm^N$ be so that  $\overline\lambda_i=\lambda_i\mod \mm^N$.  If  $\overline\lambda_{j+1}=\ldots=\overline\lambda_r=0$, then  $\delta_{j+1}\in {\mathcal D}_{\mathfrak m}(\delta_1,\ldots,\delta_j)+\mm^N\Der_k(A)$. But then by  Lemma \ref{Lem-Saltos finitos paso 1},     either  $\delta_{j+1}$ leaps at $s_i$ with  $i\in \{1,\ldots, j\}$,  or  it does not leap at all, contradicting, in both cases,    the hypotheses on $\delta_{j+1}$. Thus we may assume that  $\overline\lambda_{j+1}\neq 0$. Hence, 
	$$
	\Der_k(A)/\mm^N\Der_k(A)=	\langle \overline\delta_1,\ldots, \overline\delta_j, \overline\eta_{j+1},\ldots, \overline\eta_r\rangle_K=\langle \overline\delta_1,\ldots, \overline\delta_j, \overline\delta_{j+1},\overline\eta_{j+2},\ldots, \overline\eta_r\rangle_K,
	$$
	and therefore,  
	$$
	\Der_k(A)={\mathcal D}_{\mathfrak m}(\delta_1,\ldots,\delta_{j+1}, \eta_{j+2}, \ldots, \eta_r)+\mm^N\Der_k(A)
	$$
	with  $\delta_i$ leaping at    $s_i$ for $i=1,\ldots, j+1$.   When  $j=r$, it follows by  Lemma \ref{Lem-Saltos finitos paso 1}  that   $\Leap_k(A)=\{s_1,\ldots, s_r\}$.
\end{proof1}

\color{black}
\section{Maximal rank of Jacobian matrices and generic generators of ideals}

We will start by fixing some notation. Along this section we will use $R$ for the ring of polynomials in $n$ variables over $k$,  $R=k[x_1,\ldots, x_n]$. Given a collection of elements  $\mathcal F=\{g_1,\ldots, g_\ell\} \subset R$, we will use   $J(\mathcal F)=J(g_1,\ldots, g_\ell)$ for the matrix 
$$
J(\mathcal F)=J(g_1,\ldots, g_\ell):=\left(\begin{array}{ccc}
	\partial_1 g_1 & \cdots & \partial_{n} g_1 \\
\vdots & \cdots & \vdots \\
\partial_1 g_{\ell} & \cdots &\partial_{n} g_{\ell} \\
\end{array}
\right),
$$
where $\partial_j g_i:=\frac{\partial g_i}{\partial x_j}$ for $j=1\ldots, n$, and $i=1\ldots, \ell$. 

\

Let  $I\subset R$  be an ideal and set    $A=R/I$. We will write  $J_\ell(A)$ to refer to the $(n-\ell)$--Fitting ideal of  $\Omega_{A/k}$, for  $\ell=0,\ldots, n$. When there is no risk of confusion we will write   $J_\ell$ instead of  $J_\ell(A)$.  When $\ell=\dim (A)$ then $J_\ell$ is the usual Jacobian ideal of $A$. For the remaining of the section we will be assuming that $I$ is a radical ideal of $R$.

\begin{Def}\label{Def-FittingNoNuloAltura}
Let  $\mathcal P\subset \Min(I)$. We will say that the condition $J^{\het}_{\mathcal P}$ holds at 
 $A$  if  $J_{\het(P)}(A)A_{P}=A_{P}$ for all $P\in \mathcal P$. When  $\mathcal P=\Min (I)$, we will simply write  $J^{\het}:=J^{\het}_{\Min(I)}$.
\end{Def}
 
\begin{nota} Observe that  the  $J^{\het}$-condition holds at $A$ if for instance  $A$ is generically smooth over $k$ (i.e., if there is a dense open set ${\mathcal U}$  of $\Spec(A)$ so that $\mathcal U\to \Spec(k)$ is smooth). 
\end{nota}

Our purpose is to prove the following result. 
 
\begin{prop}\label{Cor-ExisteSistemaGeneradores}
Let  $k$ be a regular ring, let $R=k[x_1,\ldots, x_n]$, let  $I\subset R$ be a radical ideal and suppose that the  condition $J^{\het}$ holds at    $A=R/I$. Let  $r:=\max \{\het(P) \ | \ P\in \Min(I)\}$ and set  $S:=R\setminus (\bigcup_{P\in \Min(I)} P)$. Then there is a system of generators   $\mathcal S$ of  $I$ from which we can select  a subset   $\mathcal F$ with  $r$ elements so that    $\rank J(\mathcal F) =\het(P)$   at $A_P$    for all  $P\in \Min(I)$. As a consequence,   $S^{-1} I=\langle \mathcal F\rangle \subset S^{-1}R$. 
\end{prop}

\begin{proof} The  first statement of the proposition follows from (the more general statement in) Proposition \ref{Lem-ExisteSistemaGen1} below. 
Taking this from granted, the second assertion   can be deduced   from a similar argument as that given  in the proof of \cite[Theorem 11]{mat-intder-I}: the hypothesis on the rank of $J(\mathcal F)$ guarantees that the image of the  linear map of  $(R/P)_P$-vector spaces 
	$$\begin{array}{rrcl} 
		\frac{\partial }{\partial X}: &  P/P^2&  \longrightarrow & (A_P)^n =((R/P)_{P})^n \\
		& f & \mapsto &   \left(\partial_1 f, \ldots,\partial_n f \right),  
	\end{array}$$ 
has maximal rank (equal to $\het(P)$) when restricted to the classes of the elements in ${\mathcal F}$. As this is the case  at all minimal primes $P$ of $I$,  the claim follows. 
\end{proof}

\begin{prop}\label{Lem-ExisteSistemaGen1}
Let  $k$ be a regular ring, set $R=k[x_1,\ldots, x_n]$, let  $I\subset R$ be a radical ideal and let $A=R/I$. Let  $\mathcal P\subseteq \Min(I)$  be a non-empty subset  so that the condition $J^{\het}_{\mathcal P}$ holds at $A$.  If  $r=\max\{\het(P)\ | \ P\in \mathcal P\}$, then there is a set of generators of $I$, $\mathcal S=\{f_1,\ldots, f_s\}$, for which there is a subset with $r$ elements,     $\mathcal F\subset \mathcal S$, such that  $\rank J(\mathcal F)= \het(P)$ at $A_P$ for all $P\in \mathcal P$. 
\end{prop}

\begin{proof} The proposition follows from the following claim: 
\begin{quote} 
	{\bf Claim.} {\em Let 
   $\mathcal S=\{f_1,\ldots, f_s\}\subset R$ be a set of generators of $I$ and suppose there is a subset $\mathcal F\subset \mathcal S$ such that  $\rank J(\mathcal F) =\min\{\sharp \mathcal F, \het(P)\}$ at $A_P$  for all  $P\in \mathcal P$. Then, under the hypothesis of the proposition, there is some  $g\in I$ such that  $\rank(J(\mathcal F\cup \{g\}))=\min\{\sharp \mathcal F+1,\het (P)\}$ at $A_P$ for all $P\in \mathcal P$. }
\end{quote}	
	
 Suppose $I=\langle g_1,\ldots, g_\rho\rangle$. Taking the claim for granted, there exists some  $f_1\in I$ so that $\rank J(f_1)=\min \{1,\het(P)\}$ at $A_P$ for all $P\in \mathcal P$. Take $\mathcal S=\{f_1, g_1,\ldots, g_\rho\}$  as a collection of generators of $I$. Now  $\mathcal F=\{f_1\}$ and  use the claim again if needed. After a finite number of iterations, the result follows. 

\medskip 

Thus, it remains to prove the claim. First of all, if  $\mathcal F=\emptyset$ then  set  $\rank(J(\mathcal F))=0$.  Let $\ell=\sharp\mathcal F$. If $\het(P)\leq \ell$ for all    $P\in \mathcal P$, then $\rank J(\mathcal F)=\het(P)$ and there is nothing to prove since the inequality  $\rank J(\mathcal F)\leq \het(P)$ always holds. 

\medskip

Otherwise, set $\mathcal Q:=\{P\in \mathcal P \ | \ \het(P)>\ell\}$  and let  $m=\sharp \mathcal Q$. We will prove the claim  by induction on $m\geq 1$.   Suppose  $m=1$, so  $\mathcal Q=\{P\}$. To ease the notation suppose  $\mathcal F=\{f_1,\ldots, f_\ell\}$ (the set could be empty if $\ell=0$).

\medskip

The hypothesis is that  $\rank J(\mathcal S)=\het(P)>\ell$ at $A_P$. Since  $\rank J(f_1,\ldots, f_\ell)=\ell$ at $A_P$  and $J(f_1,\ldots, f_\ell)$ is a submatrix of    $J(\mathcal S)$, there must be a minor of size  $\ell$ in  $J(\mathcal S)$ determined by the first $\ell$ rows  which does not vanish at   $A_P$. Then, there must be some   $(\ell+1)$-minor in  $J(\mathcal S)$  which does not vanish at   $A_P$, and without loss of generality we can assume that it is determined by the first   $\ell+1$ rows. Setting  $g=f_{\ell+1}$ the result follows. 

\medskip

Now assume the claim holds whenever $0\leq \sharp  \mathcal Q<m$  and suppose  that    $\sharp  \mathcal Q=m>1$. To ease the notation set $\mathcal F=\{f_1,\ldots, f_\ell\}$  and  $\mathcal Q=\{P_1,\ldots, P_m\}$. 

\medskip

Define  $\mathcal P^1=\mathcal P\setminus\{P_1\}$ and  $\mathcal P^m=\mathcal P\setminus \{P_m\}$. Then  $\mathcal Q_1:=\{P\in \mathcal P^1 \ | \ \het(P)>\ell\}=\{P_2,\ldots, P_m\}$  and   $\mathcal Q_m:=\{P\in \mathcal P^m \ | \ \het(P)>\ell\}=\{P_1,\ldots, P_{m-1}\}$. From the inductive hypothesis, there is some    $h_m\in I$ such that  $\rank(J(f_1,\ldots, f_\ell, h_m))=\min\{\ell+1,\het(P)\}$ at $A_P$ for all   $P\in \mathcal P^1$ and there is some  $h_1\in I$ such that  $\rank(J(f_1,\ldots, f_\ell, h_1))=\min\{\ell+1,\het(P)\}$ at  $A_P$ for all  $P\in \mathcal P^m$.

\medskip

If either $\rank(J(f_1,\ldots, f_\ell, h_m))=\min\{\ell+1,\het(P_1)\}$ at $A_{P_1}$ or  $\rank(J(f_1,\ldots, f_\ell, h_1))=\min\{\ell+1,\het(P_m)\}$  at $A_{P_m}$, then setting   $g=h_m$ or  $g=h_1$ the result follows.  Otherwise, 
$$
\begin{array}{l}
	\rank(J(f_1,\ldots, f_\ell, h_m))=\left\{\begin{array}{l}\min\{\ell+1,\het(P)\} \mbox{ at } A_{P} \ \forall P\in \mathcal P^1\\
		\ell\mbox{ en } A_{P_1}\end{array}\right.\\
	\rank(J(f_1,\ldots, f_\ell, h_1))=\left\{\begin{array}{l}\min\{\ell+1,\het(P)\} \mbox{ at } A_{P} \ \forall P\in \mathcal P^m\\
		\ell\mbox{ en } A_{P_m}.\end{array}\right.
\end{array}
$$
Take  $\lambda\in (P_2\cap \ldots\cap P_m)\setminus P_1$  and let  $g:=h_m+\lambda h_1\in I$. Then  $\rank J(f_1,\ldots, f_\ell,g) =\min\{\ell+1, \het(P)\}$ at $A_P$ for all  $P\in \mathcal P$. This is straightforward  for $P\not\in \mathcal Q$. And if $P_i\in \mathcal Q$, then  set   $\mathcal C\subset \{1,\ldots, n\}$, with  $\sharp \mathcal C=\ell+1$, so that the  $(\ell+1)$-minor of the matrix $J(f_1,\ldots, f_\ell,h_1)$   determined by the columns of  $\mathcal C$  if $i=1$, or  $J(f_1,\ldots, f_\ell,h_m)$ if  $i\neq 1$, does not vanish at  $A_{P_i}$. From here  it follows that the   $(\ell+1)$-minor of  $J(f_1,\ldots,f_\ell,g)$  determined by  $\mathcal C$ is non-zero at  $A_{P_i}$ for  $i=1,\ldots, m$.

\end{proof}

\begin{nota} \label{series_jacobiano}
	Observe that Propositions \ref{Cor-ExisteSistemaGeneradores} and  \ref{Lem-ExisteSistemaGen1} also hold when replacing   $R=k[x_1,\ldots, x_n]$ by $\widetilde{R}:=k\lbr x_1,\ldots, x_n\rbr$. To be more precise, notice first the universal finite module  of differentials of $\widetilde{R}$ over $k$, $\widetilde{\Omega}_{\widetilde{R}/k}$,  exists,  and,  in fact, $\widetilde{\Omega}_{\widetilde{R}/k}=\oplus_{r=1}^n\widetilde{R}dx_i$. Moreover,   the universal derivative, $\widetilde{d}: \widetilde{R}\to \widetilde{\Omega}_{\widetilde{R}/k}$, is the continuous extension of the derivative $d: R\to  \Omega_{R/k}$. 
	From the previous discussion it follows that    $\widetilde{\Omega}_{\widetilde{R}/k}$ is a finite free $\widetilde{R}$-module and for $f(x_1,\ldots, x_n)\in \widetilde{R}$, 
	$$\widetilde{d}(f(x_1,\ldots, x_n))=\sum_{i=1}^n\frac{\partial f}{\partial x_i}dx_i,$$
	see 	\cite[Example 12.7]{Kunz}. 
	Now let $I\subset \widetilde{R}$ be an ideal and set $\widetilde{A}=\widetilde{R}/I$. Then by \cite[Corollary 11.10]{Kunz}, the universal finite module of differentials of $\widetilde{A}$ over $k$, $\widetilde{\Omega}_{\widetilde{A}/k}$,  exists, and there is an exact sequence, 
	$$I/I^2\stackrel{\widetilde{d}}{\longrightarrow} \widetilde{\Omega}_{\widetilde{R}/k}/I\widetilde{\Omega}_{\widetilde{R}/k} \longrightarrow  
	\widetilde{\Omega}_{\widetilde{A}/k}\to 0,$$
	which in turns gives a presentation of $\widetilde{\Omega}_{\widetilde{A}/k}$.  
	By \cite[Appendix D]{Kunz}, the Fitting ideals of $\widetilde{\Omega}_{\widetilde{A}/k}$ are defined in the same way as those of ${\Omega}_{{A}/k}$. Hence we can define the ideals $J_{\ell}(\widetilde{A})$ as the $(n-\ell)$-Fitting ideal of $\widetilde{\Omega}_{\widetilde{A}/k}$ 
	in the same way as we did for $\Omega_{A/k}$. 
\end{nota}

	\section{Jacobians and  integrability}

\color{black} In this section we will give some sufficient conditions for a derivation  to be $\infty$-integrable in the case of certain classes of  $k$-algebras  of finite type, $A=k[x_1,\ldots, x_n]/I$.  For the first result we do not need to impose extra conditions on $k$;  in exchange, a (strong) condition on the Jacobian ideal of $A$ is required.  For the other two, we need    $k$  to be regular. Our results extend those of  H. Matsumura \cite[Theorem 11, Corollary 1]{mat-intder-I} and  L. Narv\'aez Macarro \cite[Proposition 2.2.4]{Na2} or \cite[Proposition 2.2.1]{Na2C}.    


\medskip

\noindent{\bf The following covers the case of complete intersections:}

\begin{teo}\label{Teo-DerIntInterseccionCompleta} Let  $k$ be a commutative ring with unit, let $R=k[x_1,\ldots, x_n]$, let  $I=\langle f_1,\ldots, f_r\rangle$  and  set $A=R/I$. Let  $\delta\in \Der_k(A)$. If    $\delta(A)\subset J_r(A)$, then  $\delta\in \IDer_k(A)$. In  particular, $J_r(A)\Der_k(A)\subset \IDer_k(A)$. 
\end{teo}

\noindent{\bf Equidimensional $k$-algebras:}

\begin{teo}\label{Teo-EQUIDIMENSIONALJACOBIANO} Let  $k$  be  a regular ring, set   $R=k[x_1,\ldots, x_n]$ and let  $I\subset R$ be a radical ideal. Suppose  $A=R/I$ is  an equidimensional ring of codimension  $r$   satisfying the condition  $J^{\het}$. Let  $\delta\in \Der_k(A)$ and let  $\Delta\in J_r(A)$ be a non-zero divisor in  $A$. Then, $\Delta \delta\in \IDer_k(A)$. Moreover, there is an  $\infty$--integral $E=(E_\mu)\in \HS_k(A)$ of  $\Delta\delta$ such that   $E_\mu(x_i)\in \langle \Delta\rangle$, for all  $\mu\geq 1$, for all $i=1,\ldots, n$.    As a consequence $J_r(A)\Der_k(A)\subset \IDer_k(A)$. 
\end{teo}

\noindent{\bf Reduced $k$-algebras:}

\begin{teo}\label{Teo-CasoNoEquidimensional}
Let $k$ be a regular ring,  set $R=k[x_1,\ldots, x_n]$ and let  $I\subset R$ be a radical ideal with  $r=\max \{\het(P)\ | \ P\in \Min(I)\}$.  Let  $A=R/I$ and  let  $J_r^0(A)$ be  some lifting of $J_r(A)$ in  $R$.  Then, $(J_r^0(A)+I)\Der_k(\log I)\subset \IDer_k(\log I)$. As a consequence, $J_r(A)\Der_k(A)\subset \IDer_k(A)$. 
\end{teo}

To address the proofs of the theorems we will make use of the following setting and notation.  
	
\begin{parr} \label{setting_notation} {\bf A common setting for the proofs of the theorems.} {\rm With the notation of the theorems, let $\delta\in \Der_k(A)$ be a derivation.  What does it mean to find a $\nu$-integral of $\delta$, for $\nu\geq 2$? In the next lines we follow the ideas in  \cite[Theorem 11]{mat-intder-I} (see also \cite[Corollary 2.1.3]{Na2}). 

\medskip

To start with,  our datum is $\delta$. Thus  $\delta(x_i)=\xi_{1i}$ for some  $=\xi_{1i}\in A$. In our case, under the hypotheses of the theorems, $\xi_{1i}\in \Ma$, where $\Ma$ is a suitable given ideal, for $i=1,\ldots, n$, and this will play a role in our arguments. Suppose, by induction, that we have been able to find a $(\nu-1)$-integral of $\delta$, i.e. a morphism of $k$-algebras:  
\begin{equation}\label{Eq-Integral}
	\varphi:x_i\in A\to x_i+\sum_{\mu=1}^{\nu-1} t^\mu \xi_{\mu i}\in A\lbr t\rbr_{\nu-1}, \ \mbox{ with } \xi_{\mu i}\in \Ma, \ \forall \mu=1,\ldots, \nu-1, \ \forall i=1,\ldots,n.
\end{equation} 
Notice that this also induces a morphism of $k$-algebras on $k[x_1,\ldots, x_n]$. In particular, for each $f_{\alpha}\in I\subset k[x_1,\ldots, x_n]$,
writing  $\xi_{\mu}:=(\xi_{\mu 1},\ldots, \xi_{\mu n})$, for  $\mu=1,\ldots, \nu-1$, and $x=(x_1,\ldots, x_n)$, we have that, in $A\lbr t\rbr_{\nu}$, 
$$f_{\alpha}\left(x+\sum_{\mu=1}^{\nu-1}t^{\mu}\xi_{\mu}\right)\equiv t^{\nu}F_{\alpha}(x)\ \ \text{mod } t^{\nu+1}.$$

\medskip

\begin{quote} {\bf Key point.} 
{\em Observe that  $F_\alpha$ is a linear combination over $A$ of monomials of the form   $\xi_{\mu_1 j_1}\xi_{\mu_2 j_2}\cdots \xi_{\mu_q j_q}$ where $\mu_1+\ldots+\mu_q=\nu$.  Since $\mu_i< \nu$, it follows that  $q\geq 2$. So, $F_\alpha\in \Ma^2$.}
\end{quote}

\medskip 

For  $\xi_{\nu 1},\ldots, \xi_{\nu n}\in A$   set $\xi_\nu:=(\xi_{\nu 1},\ldots, \xi_{\nu n})$. Then  
$$
f_{\alpha}\left(x+\sum_{\mu=1}^{\nu}t^{\mu}\xi_{\mu}\right)\equiv t^{\nu}\left[F_{\alpha}(x)+\sum_{j=1}^n\partial_j \left(f_{\alpha}\right)\xi_{\nu j}\right] \ \ \text{mod } t^{\nu +1}.
$$

Thus, to  find a $\nu$-integral of $\delta$ it suffices to find $\xi_\nu:=(\xi_{\nu 1},\ldots, \xi_{\nu n})$ so that 
\begin{equation}
	\label{eq_4_1}
	F_{\alpha}+\sum_{j=1}^n\partial_j \left(f_{\alpha}\right)\xi_{\nu j}=0 \ \in  A, 
\end{equation}
for all $f_\alpha\in I$. In other words, we need to find a solution  $\xi_\nu=(\xi_{\nu 1}, \ldots, \xi_{\nu n})$ to the linear system whose augmented matrix is:  
\begin{equation}
	\label{sistema_lineal_general}
	\left(\begin{array}{ccc|c} 
		\partial_1 f_1 & \ldots & \partial_n f_1 & -F_1(x)\\
		\vdots & \ldots & \vdots & \vdots  \\
		\partial_1 f_s & \ldots & \partial_n f_s & -F_{s}(x)
	\end{array}\right),
\end{equation}
where $\langle f_1,\ldots, f_s\rangle =I$. We will see that the assumptions of the theorems guarantee the existence of a solution and that, furthermore, this solution  can be chosen such that $\xi_{\nu j}\in \Ma$ for all $j=1,\ldots, n$.   We will make use of the following lemma. }
\end{parr}

		\begin{lem}\label{Lema-SistemaSolucion}
	Let $B$ a commutative ring with unit, and let  $M\in \M_{s\times n}(B)$ and  $\bfb\in \M_{s\times 1}(B)$ be two  matrices. Suppose that  $\rank(M|\bfb)=r$ and let  $\Delta\in B$ be an  $r$-minor of $M$.  Then the system   $M\bfx=\Delta\bfb $ has a solution in $B$. Moreover,  there is a solution  $\xi=(\xi_1,\ldots, \xi_n)\in B^n$ so that  $\xi_i\in \langle b_1,\ldots, b_s\rangle$  for  $i=1,\ldots, n$. 
	\end{lem}
	
	\begin{proof} To ease the notation let us assume that the submatrix defined by   $\Delta$ is given by the first $r$ rows and the first $r$ columns of $M$. For each  $i=1,\ldots, s$, define   $M_i\in \M_{(r+1)\times (r+1)}(A)$ as follows: 
		$$
		M_i=\left(\begin{array}{cccccccccc}
			m_{11} &\ldots&m_{1r}&b_1\\
			\vdots & \ldots& \vdots\\
			m_{r1}&\cdots&m_{rr}&b_r\\
			m_{i1}&\cdots&m_{ir}&b_i
		\end{array}
		\right).
		$$
Notice that for     $j=1,\ldots, r$, the cofactors corresponding to the last row are so that 
  $$\text{Cof} ((M_i)_{r+1,j})=\text{Cof}((M_{i'})_{r+1,j})$$ for all  $i\neq i'$. In addition,   $\text{Cof} ((M_i)_{r+1,r+1})=\Delta$  for all $i$. Set  $\xi_j=\text{Cof}((M_i)_{r+1,j})\in \langle b_1,\ldots, b_r\rangle\subset \langle b_1,\ldots, b_n\rangle$ for  $j=1,\ldots, r$.  Since  $|M_i|=0$ for  $i=1, \ldots, s$, we have that 
		$$
		0=|M_i|=m_{i1}\xi_1+\ldots+m_{ir}\xi_r+b_i\Delta. 
		$$
		Hence,  $\xi=(-\xi_1,\ldots, -\xi_r,0, \ldots,0)$ is a solution with the properties stated in the lemma (here we assign zero to  the variables corresponding  to the columns not appearing in  $\Delta$). 
	\end{proof}

\medskip

\begin{proof1}{Theorem}{\ref{Teo-DerIntInterseccionCompleta}} If $J_r:=J_r(A)=0$, then $\delta=0$ and thus  $\infty$-integrable. Otherwise, $\rank J(f_1,\ldots, f_r)=r$. The hypotesis on $\delta$ implies that $\delta(x_i)=\xi_{1i}\in J_r$ for all $i=1,\ldots, n$. Suppose by induction that we have a $(\nu-1)$-integral of $\delta$ as (\ref{Eq-Integral}) where $\Ma=J_r$. Notice that in this case, $s=r$ in (\ref{sistema_lineal_general}) and $F_\alpha \in J_r^2$ (see the Key point in \S \ref{setting_notation}). Let   $\Delta_1,\ldots, \Delta_a$  be the non-zero $r$-minors of $J(f_1,\ldots, f_r)$. Since  $F_\alpha\in J_r^2$,  
		\begin{equation}
			\label{F_i_menor_igual_r}
			F_\alpha=\sum_{\lambda}\Delta_{\lambda} h_{\alpha\lambda}, \ \ h_{\alpha\lambda}\in J_r\subset A\ \ (1\leq \alpha \leq r). 
		\end{equation}
For  $1\leq \lambda\leq a$, the following matrix has rank $r$, 
		$$
		\left( \begin{array}{ccc|c}
			\partial_1 f_1& \ldots & \partial_n f_1 & h_{1\lambda}\\
			\vdots & \vdots & \vdots & \vdots\\
			\partial_1 f_r& \ldots & \partial_n f_r &  h_{r\lambda}\\
		\end{array}
		\right).$$
	  Hence, by Lemma \ref{Lema-SistemaSolucion}, the system 
		$$\left( \begin{array}{ccc}
			\partial_1 f_1 & \ldots & \partial_n f_1\\
			\vdots & \vdots & \vdots \\
			\partial_1 f_{r}& \ldots &\partial_n f_{r}
		\end{array}
		\right) \left(\begin{array}{c} x_{1}\\ \vdots \\  x_{ n}\end{array}\right)=\Delta_\lambda\left(   \begin{array}{c} h_{1\lambda} \\ \vdots \\ h_{r\lambda}  \end{array}   \right)$$
		has a solution    $\xi^{(\lambda)}_{\nu}=(\xi^{(\lambda)}_{\nu1}, \ldots, \xi^{(\lambda)}_{\nu n})$ for every   $\lambda=1,\ldots, a$, and in addition   $\xi^{(\lambda)}_{\nu j}\in \langle h_{1\lambda}, \ldots, h_{r\lambda}\rangle\subset J_r$.  Therefore  $\xi_\nu:=\sum_{\lambda} -\xi^{(\lambda)}_{\nu}\in (J_r)^n$  is a solution of  (\ref{eq_4_1}) and the result follows from the equality in   (\ref{eq:limit-finite}). 
	\end{proof1}
	
\medskip

\begin{proof1}{Theorem}{\ref{Teo-EQUIDIMENSIONALJACOBIANO}}
Assuming that  the first statement of the theorem holds, the second follows immediately: from the hypotheses, $J_r:=J_r(A)$ is so that $J_rA_{P}\neq 0$ for  $P\in \Min(I)$, and then, since $I$ is a radical ideal, it can be checked that  $J_r$ can be generated by a finite collection of non-zero divisors in $A$, say $J_r=\langle \Delta_1,\ldots, \Delta_a\rangle$. Since we are assuming that the first statement holds,    $\Delta_\lambda\Der_k(A)\subset \IDer_k(A)$  for  $\lambda=1,\ldots, a$, hence $J_r\Der_k(A)\subset \IDer_k(A)$. 
		
		\smallskip
	To prove the first statement of the theorem, let   $\mathcal S=\{f_1,\ldots, f_s\}$ be a collection of generators of    $I$ so that  $S^{-1} I=\langle f_1,\ldots, f_r\rangle$ where  $S=R\setminus \left(\cup_{P\in \Min(I)} P\right)$ 		 (see Proposition  \ref{Cor-ExisteSistemaGeneradores}) and let  $D:=\Delta \delta\in \Der_k(A)$. Hence, $D(x_i)=\xi_{1i}\in \langle \Delta\rangle$ and, from \S\ref{setting_notation}, after an inductive argument, suppose that we have a $(\nu-1)$-integral of $D$ as (\ref{Eq-Integral}) where $\Ma=\langle \Delta\rangle$. 
Observe that  $F_\alpha\in \langle \Delta^2\rangle$, i.e.,
		$
		F_\alpha =\Delta h_\alpha, \ \mbox{ with } h_\alpha\in \langle \Delta\rangle, \ \ (1\leq \alpha \leq s)
		$ (see the Key point in \S\ref{setting_notation}) and recall that  we would like to find a solution  $\xi_{\nu j}\in \langle \Delta\rangle$ ($1\leq j\leq n$) of the system (\ref{sistema_lineal_general}).	
		Since $S^{-1}I=\langle f_1,\ldots, f_r\rangle$, for  $1\leq q\leq s-r$, we have that 
 $$
		f_{r+q}=\sum_{i=1}^r a_{qi} f_i, \ \  a_{qi}\in S^{-1}R,
		$$
	and hence, 
		$$
		\partial_j f_{r+q}=\sum_{i=1}^r a_{qi} \partial_jf_i, \ \ 
		F_{r+q}=\sum_{i=1}^r  a_{qi} F_i, \ \ \mbox{ at } S^{-1}A. 
		$$
Since 		 $\Delta$ is a non-zero divisor, i.e., $\Delta\in S$,  
		$$
		h_{r+q}= \sum_{i=1}^ra_{qi} h_i, \mbox{ at } S^{-1}A.
		$$
	Now, the matrix:
		$$M:=\left(\begin{array}{ccc|c} 
			\partial_1 f_1 & \ldots & 	\partial_n f_1 & h_1\\
			\vdots & \ldots & \vdots & \vdots  \\
			\partial_1 f_s& \ldots & \partial_n f_s & h_{s}
		\end{array}\right)\in \M_{s\times (n+1)}(A)$$
	has rank $r$ at  $S^{-1}A$, and hence, it also has rank $r$ at $A$ (since there are no zero divisors in $S$). On the other hand,  $J_r$ is generated by the non-zero $r$-minors of   $J(f_1,\ldots, f_s)$, say $J_r=\langle \Delta_1,\ldots, \Delta_a\rangle$. Hence,  $\Delta=\gamma_1\Delta_1+\ldots+\gamma_a\Delta_a$ with $\gamma_\lambda\in A$. Now observe  that, for $\lambda=1,\ldots, a$, the following matrix has rank $r$, 
		$$
		\left(\begin{array}{ccc|c} 
			\partial_1 f_1 & \ldots & 	\partial_n f_1 & \gamma_\lambda h_1\\
			\vdots & \ldots & \vdots & \vdots  \\
			\partial_1 f_s& \ldots & \partial_n f_s & \gamma_\lambda h_{s}
		\end{array}\right)\in \M_{s\times (n+1)}(A). 
		$$
		
		Thus, by Lemma \ref{Lema-SistemaSolucion}, for   $\lambda=1,\ldots, a$,  the system
		$$
		J(f_1,\ldots, f_s)\bfx= \Delta_\lambda \gamma_\lambda \left(\begin{array}{c} h_1\\
			\vdots\\
			h_s
		\end{array}\right)
		$$
		has a solution  
 $\xi^{(\lambda)}=(\xi^{(\lambda)}_1,\ldots, \xi^{(\lambda)}_n)$ so that   $\xi^{(\lambda)}_j\in \langle \gamma_\lambda h_1, \ldots, \gamma_\lambda h_s\rangle\subset \langle \Delta\rangle$, for  $j=1,\ldots, n$. Therefore,   $\xi_{\nu}=-\sum_{\lambda=1}^a \xi^{(\lambda)}$ is a solution of the system (\ref{sistema_lineal_general}), i.e.,  we have a  $\nu$-integral of  $\Delta\delta$ given by
$$
\varphi':x_i\in A\to x_i+\sum_{\mu=1}^\nu t^\mu \xi_{\mu i}\in A\lbr t\rbr_{\nu} \ \mbox{ where } \xi_{\mu i}\in \langle \Delta\rangle.
$$
To conclude, by (\ref{eq:limit-finite}), there is an  $\infty$-integral $E\in \HS_k(A;\infty)$ of  $\Delta\delta$ so that  $E_\mu(x_i)=\xi_{\mu i}\in \langle \Delta\rangle$. 
	\end{proof1}

\medskip

Theorem \ref{Teo-EQUIDIMENSIONALJACOBIANO}  can be restated in a logarithmic version. We include it here, since it will be useful to address the proof of Theorem \ref{Teo-CasoNoEquidimensional}. 
	
	\begin{cor}\label{Cor-EquidimansionalLogaritmico}
		Let  $k$ be a regular ring, set $R=k[x_1,\ldots, x_n]$ and let  $I\subset R$ be a radical ideal. Suppose $A=R/I$ is equidimensional of codimension $r$,  and assume   it satisfies the condition  $J^{\het}$. Let  $J_r^0(A)$ be a lifting of   $J_r(A)$ in $R$. Let  $\delta\in \Der_k(\log I)$ and let  $\Delta\in J_r^0(A)$ be a non-zero divisor in   $A$. Then, $\Delta \delta\in \IDer_k(\log I)$. In addition, there exists $E\in \HS_k(\log I)$ an $\infty$-integral of $\Delta\delta$ so that  $E_\mu(x_i)\in \langle \Delta\rangle$ for all $\mu\geq 1$, for all $i=1,\ldots, n$.  As a consequence,  $(J_r^0(A)+I)\Der_k(\log I)\subset \IDer_k(\log I)$. 
	\end{cor}

\medskip
	
\begin{proof1}{Theorem}{\ref{Teo-CasoNoEquidimensional}} Let  $r:=\max\{\het(P)\ | \ P\in \Min(I)\}$ and let  $J_r^0:=J_r^0(A)$ be a lifting in  $R$ of  $J_r(A)$. Consider $\Delta\in J_r^0$ and  $\delta\in \Der_k(\log I)$. We will show that  $\Delta\delta\in \IDer_k(\log I)$, which will imply that  $J_r^0\Der_k(\log I)\subset \IDer_k(\log I)$. Since  $I\cdot \Der_k(\log I)\subset \IDer_k(\log I)$,  the theorem follows. 
		
	Set $\mathcal P:=\{Q\in \Min(I)\ | \ \Delta\not\in Q\}=\{P_1,\ldots, P_m\}$,  
	   write $I_1=P_1\cap \ldots \cap P_m$ and  $B=R/I_1$. If  $\mathcal P=\emptyset$, then  $\Delta\in I$ and the integrability is immediate.  Otherwise, since   $J_rA_Q=0$  for all  $Q\in \Min(I)$ with  $\het(Q)<r$, we have that  $\het(P_i)=r$ for  $i=1,\ldots, m$. In addition, it can be checked that   $\Delta\in J_r^0(B)$, where  $J_r^0(B)$ is a lifting in   $R$ of  $J_r(B)$, and hence  $J_r(B)B_{P_i}\neq 0$. Thus,  $B$ is  equidimensional and verifies the  condition  $J^{\het}$. By Corollary \ref{Cor-EquidimansionalLogaritmico},  we have that  $(J_r^0(B)+I_1)\Der_k(\log I_1)\subset \IDer_k(\log I_1)$. 
	
	 By Lemma \ref{Lem-AsociadosDerivaciones}, $\Der_k(\log I)\subset \Der_k(\log I_1)$. Hence  $\Delta\delta\in (J_r^0(B)+I_1)\Der_k(\log I_1)\subset \IDer_k(\log I_1)$. Since  $\Delta$ is a non-zero divisor in $B$, by  Corollary \ref{Cor-EquidimansionalLogaritmico}, there is an  $\infty$-integral  $E\in \HS_k(\log I_1)$  of  $\Delta\delta$ such that  $E_{\mu}(x_i)\in \langle \Delta\rangle$. If   $Q\in\Min(I)\setminus \mathcal P$,  $\Delta\in Q$ and hence,  $E\in \HS_k(\log Q)$. By  Lemma \ref{Lem-HSInterseccionIdeales},
		$
		E\in \HS_k(\log I)
		$,
		from where it follows that  $\Delta\delta\in \IDer_k(\log I)$.  It can be checked that the theorem follows from here. 
	\end{proof1}

\begin{nota}\label{series_integrabilidad}  Theorems \ref{Teo-DerIntInterseccionCompleta}, \ref{Teo-EQUIDIMENSIONALJACOBIANO}, \ref{Teo-CasoNoEquidimensional} also hold  for $\widetilde{R}=k\lbr x_1,\ldots, x_n\rbr$ and $\widetilde{A}=\widetilde{R}/I$ with  $I\subset \widetilde{R}$  when keeping the corresponding assumptions on $k$ and  $\widetilde{A}$. This follows from Remark \ref{series_jacobiano} and the fact that in the discussion in \S\ref{setting_notation}, the map 
	$\widetilde{A}\to \widetilde{A}[|t|]$ in (\ref{Eq-Integral}) is continuous with respect to the $\langle x_1,\ldots, x_n\rangle$-adic topology.  
	
\end{nota}

	\section{Rings with a finite number of leaps}

In section \ref{SEC-FinitudSaltosLocales}, we gave a sufficient condition for a local ring to have a finite number of leaps. In this section, we prove the next result:

\begin{teo}\label{Teo-AnillosConSaltosFinitos} Let $k$ be a Noetherian ring containing a field, let $R=k[x_1,\ldots, x_n]$ and let  $I\subset R$ be an  ideal. Set  $A=R/I$ and let $\Mp\in \Spec(A)$ a minimal prime of  $J_r:=J_r(A)$, the    $(n-r)$--Fitting ideal of  $\Omega_{A/k}$.  Suppose at least one of the following conditions hold: 
	\begin{itemize}
		\item[1)] $I=\langle f_1,\ldots, f_r\rangle$; 
		\item[2)] $k$ is regular, $I$ is radical and  $r=\max\{\het(P) \ | \ P\in \Min(I)\}$. 
	\end{itemize}
Then the set $\Leap_k(A_{\Mp})$ is finite of cardinal bounded by $d:=\dim_K (\Der_k(A_\Mp)/\Mp^M \Der_k(A_\Mp))$ where $K$ is the residue field of   $\Mp$,  and $M$ is the smallest positive integer so that  $\Mp^M\Der_k(A_\Mp)\subset \IDer_k(A_\Mp)$.  
\end{teo}
	
\begin{proof}
By Theorems  \ref{Teo-DerIntInterseccionCompleta} and  \ref{Teo-CasoNoEquidimensional} respectively,   $J_r\Der_k(A)\subset \IDer_k(A)$. Since  $A$ is Noetherian and finitely presented over $k$,  by  Lemma \ref{Cor-SJcDer infinito-integrable1}, $\Mp^N\Der_k(A_\Mp)\subset \IDer_k(A_\Mp)$,  where  $N$ is the smallest positive integer for which    $\Mp^N\subset J_r A_\Mp$. To conclude,  observe that  $(A_\Mp, \Mp A_{\Mp}, K)$ is under the hypotheses of   Theorem \ref{Prop-Saltos finitos} for some  $M\geq 1$ such that  $\Mp^M\Der_k(A_{\Mp})\subset \IDer_k(A_{\Mp})$, with $\Der_k(A_\Mp)$ a finitely generated $A_\Mp$-module,  from where it follows that $\Leap_k(A_\Mp)$ is finite and of cardinal  less than or equal to  $d$.
	\end{proof}

	\begin{cor}\label{Cor-SingularidadAislada}
		Let  $k$ be a perfect field,  let $R=k[x_1,\ldots,x_n]$ and let  $I\subset R$ be a radical ideal so that  $A=R/I$ is equidimensional of codimension  $r$. Suppose that  $\Sing(A)=\mathbb V(J_r)=\{\mm_1,\ldots,\mm_s\}\subset \Specmax(A)$, where $J_r$ is   the Jacobian ideal of $A$.  Then  $\Leap_k(A)$  is a finite set. 
	\end{cor}
\begin{proof}
By Proposition \ref{Prop-UnionSaltos}, we have that 
$$
\Leap_k(A)=\bigcup_{\Mp\in \Sing(A)} \Leap_k(A_{\Mp})=\Leap_k(A_{\mm_1})\cup \ldots \cup \Leap_k(A_{\mm_s}).
$$
Now Theorem  \ref{Teo-AnillosConSaltosFinitos} guarantees that  $\Leap_k(A_{\mm_i})$ is finite for  $i=1,\ldots, r$.  
\end{proof}
	
\begin{cor} \label{corolario_curvas}
Let   $k$ be  a perfect field and let $A$ be a reduced $k$-algebra of finite type of  dimension 1. Then, $\Leap_k(A)$ is finite. 
\end{cor}

\begin{teo}\label{Teo-AnillosConSaltosFinitosFormal} Let $k$ be a Noetherian ring containing a field, let $\widetilde{R}=k\lbr x_1,\ldots, x_n\rbr$ and let  $I\subset \widetilde{R}$ be an  ideal. Set  $\widetilde{A}=\widetilde{R}/I$  and let 
  $J_r:=J_r(\widetilde{A})$  be  the    $(n-r)$--Fitting ideal of  $\widetilde{\Omega}_{\widetilde{A}/k}$.  Suppose that the radical of $J_r$ is a maximal ideal ${\mathfrak m}\subset \widetilde{A}$ and that, in addition,  at least one of the following conditions hold: 
	\begin{itemize}
		\item[1)] $I=\langle f_1,\ldots, f_r\rangle$; 
		\item[2)] $k$ is regular, $I$ is radical and  $r=\max\{\het(P) \ | \ P\in \Min(I)\}$. 
	\end{itemize}
	Then the set $\Leap_k(\widetilde{A})$ is finite of cardinal bounded by $d:=\dim_K (\Der_k(A)/{\mathfrak m}^M \Der_k(\widetilde{A}))$ where $K$ is the residue field of   ${\mathfrak m}$,  and $M$ is the smallest positive integer so that  ${\mathfrak m}^M\Der_k(\widetilde{A})\subset \IDer_k(\widetilde{A})$.  
\end{teo}
\begin{proof} As indicated in Remark \ref{series_integrabilidad},   Theorems  \ref{Teo-DerIntInterseccionCompleta} and  \ref{Teo-CasoNoEquidimensional} are also valid for $\widetilde{A}$ and hence    $J_r\Der_k(\widetilde A)\subset \IDer_k(\widetilde A)$. Since  $A$ is Noetherian,  ${\mathfrak m}^N\Der_k(\widetilde A)\subset \IDer_k(\widetilde A)$,  where  $N$ is the smallest positive integer for which    ${\mathfrak m}^N\subset J_r$.  Now use Theorem \ref{Prop-Saltos finitos}. \end{proof}

\

%
%


\end{document}